\documentclass[openright,a4paper,american,12pt]{amsart}

\usepackage[latin1]{inputenc}
\pdfoutput=1
\usepackage{amsmath}
\usepackage{amssymb}
\usepackage{hyperref}
\usepackage{babel}
\usepackage{enumerate}

\usepackage{amsfonts}
\input xy
\xyoption{all}
\usepackage[dvips]{graphicx}
\usepackage{boxedminipage}
\usepackage{color}

\newcommand{\Z}{{\mathbb Z}}
\newcommand{\NN}{{\mathbb N}}

\newcommand{\CC}{{\mathbb C}}

\newcommand{\R}{{\mathbb R}}

\newcommand{\TT}{{\mathbb T}}
\newcommand{\T}{{\mathbb T}}
\newcommand{\TP}{{\mathbb{TP}}}
\newcommand{\CP}{{\mathbb{CP}}}

\newcommand{\F}{{\mathcal F}}
\newcommand{\X}{{\mathcal X}}

\renewcommand{\H}{{\mathcal H}}
\renewcommand{\S}{{\mathcal S}}

\newcommand{\Vol}{\text{Vol}}

\newtheorem{thm}{Theorem}[section]

\newtheorem{definition}[thm]{Definition}

\newtheorem{proposition}[thm]{Proposition}
\newtheorem{lemma}[thm]{Lemma}

\newtheorem{corollary}[thm]{Corollary}

          {\theoremstyle{definition}
}
          {\theoremstyle{definition}

\newtheorem{example}[thm]{Example}}

\newcommand{\comment}[1]{}
\newcommand{\overlim}[1]{{\buildrel{#1}\over\longrightarrow\;}}

\begin{document}
\title{Tropical $(1,1)$-homology for floor decomposed surfaces}
\author{Kristin M. Shaw}
\date{\today}

\email{shawkm@math.utoronto.ca}

\maketitle
\begin{abstract} 
The tropical $(p, q)$-homology groups of Itenberg, Katzarkov, Mikhalkin and Zharkov \cite{ItKaMiZh} are the tropical analogues of the Hodge decomposition of the cohomology of complex algebraic varieties. Following ideas of Mikhalkin, in \cite{KMS} it is shown that there is a well-defined intersection pairing on tropical $(1, 1)$-classes of a compact non-singular tropical surface. 
Here we compute directly the $(1, 1)$-homology of a non-singular floor decomposed tropical surface in tropical projective space, along with the intersection form.  
\end{abstract}

\begin{section}{Introduction}
Tropical geometry is a relatively new area of mathematics which studies polyhedral complexes equipped with integer affine structures, known as tropical varieties. One of the goals of tropical geometry is to study classical algebraic geometry via these tropical varieties. 
This has had powerful applications to the study of real and complex enumerative geometry in the case of curves in toric surfaces, due to Mikhalkin's correspondence theorem \cite{Mik1}. This theorem is in part a consequence of the fact that every tropical hypersurface in $\R^n$ can be approximated by a  limit of so-called amoebas of complex algebraic hypersurfaces in $(\CC^{\ast})^n$ \cite{V9},
\cite{Mik12}, 
\cite{Rullgard1}. There have also been generalisations of Mikhalkin's theorem for curves of higher codimension, see for examples \cite{Mik3}, \cite{Spe2}, \cite{NS}, \cite{Mik08}, \cite{Nishinou},
 \cite{Tyomkin},  \cite{Kat1}, \cite{Br6b}, \cite{Br9}.  As these generalisations suggest,  some tropical varieties in $\R^n$ may be approximated by limits of amoebas of $1$-parameter families of complex algebraic subvarieties of $(\CC^{\ast})^n$. Under suitable conditions this so-called tropicalisation retains information about the original variety. For example, enumerative invariants can sometimes be recovered, as well as in some cases intersection properties of  subvarieties  \cite{Pay1}, \cite{Rab2},  \cite{Br16}.

Introduced by Itenberg, Katzarkov, Mikhalkin and Zharkov, tropical $(p, q)$-homology provides a way of studying the topology of complex algebraic varieties via their tropicalisations. 
Tropical $(p, q)$-homology is  the homology 
of a tropical variety with coefficients in a certain co-sheaf, \cite{ItKaMiZh}. In other words, the $(p, q)$-tropical homology groups are $$H_{p, q} (X) = H_q(X, \F_p),$$ where the definition of $\F_p$ from \cite{ItKaMiZh} will be recalled  in Section \ref{sec:framings}.
The tropical $(p, q)$-homology groups are meant to be the tropical analogue of the Hodge decomposition of the cohomology of a complex algebraic variety. In particular,  assuming a family $\X_t$ and its tropicalisation  satisfy the appropriate conditions from \cite{ItKaMiZh}, the Hodge numbers $h^{p, q}$ of a member of the family $\X_t$
are given by the ranks of the tropical  $(p, q)$-homology groups. 
The reader is directed to  \cite{ItKaMiZh} for the precise relationship between the tropical homology groups and the mixed Hodge structure of the family $\X_t$.  

For a non-singular complex projective surface $\X$, the intersection form on 
$H^{1, 1}(\X)$ satisfies the Hodge index theorem. 
\begin{thm}[Hodge index theorem] \cite{Hod}
	Let $\X \subset \CP^N$ be a non-singular projective surface and  $\H \in H^{1, 1}(\X)$ denote the class of the	 	hyperplane section. Then, $$H^{1, 1}(\X) = <\H> \oplus <\H>^{\bot}$$ where the intersection form on $H^{1, 1}		(\X)$ restricted to $<\H>^{\bot}$ is negative definite.  
\end{thm} 

It follows directly from the Hodge index theorem that the signature of the intersection form on $H^{1, 1}$ is $1-h^{1, 1}(\X)$, where $h^{1, 1}(\X) =  \dim (H^{1, 1}(\X)).$ 

Following ideas of Mikhalkin, in \cite{KMS}, it is shown that on a compact non-singular tropical surface $X$ there is an intersection product on $H_{1, 1}(X)$ when it is torsion free. 
In Section \ref{sec:H11},  we show the following for the tropical $(1, 1)$-homology groups. 

\begin{thm}\label{thm:H11floor}
A smooth tropical  floor decomposed surface $X_d \subset \TP^3$ of degree $d$  has 
$$h_{1, 1}(X_d)  = \frac{2d^3 - 6d^2 + 7 d}{3}.$$
Moreover, the signature of the intersection form on $H_{1, 1}(X_d)$ is 
$$(1+ b_2(X_d), h_{1, 1}(X_d) - 1 - b_2(X_d)).$$
\end{thm}

Note that the second Betti number of a non-singular tropical hypersurface of degree $d$ is equal to the number of interior lattice points of the standard simplex of size $d$, $\Delta_d \subset \R^3$
$$b_2(X) =  \frac{d^3 -6d^2 +11d - 6}{6}.$$
Above the signature of the intersection form is denoted $(p, n)$, where $p$ is rank of the positive definite part and $n$ the rank of the negative definite part of the form.  
Indeed, the rank coincides with the rank of $H^{1, 1}$ for a non-singular complex hypersurface of $\CP^3$ of the same degree.

Sections \ref{sec:prelim} and \ref{sec:tropHomo} set up the necessary preliminary definitions starting with tropical hypersurfaces and floor decomposed  tropical surfaces in Section \ref{sec:prelim}.  Section \ref{sec:tropHomo},  outlines the definitions of tropical homology from \cite{ItKaMiZh} in the cases we consider and the intersection product defined on $H_{1, 1}$.
For more details and the full generality of the definitions the reader is referred to \cite{ItKaMiZh}, and \cite{KMS}. 
Section \ref{sec:homo} contains the computations of both the tropical homology group $H_{1, 1}$ and its intersection form. 
Firstly Section \ref{sec:basis} constructs a collection of cycles on a non-singular surface $X_d$. Using induction on the degree of the surface and a Mayer-Vietoris argument it is shown that these cycles form a basis of $H_{1, 1}(X_d)$. Finally using this basis we determine the signature of the intersection form on $H_{1, 1}(X_d)$.  The proof of Theorem \ref{thm:H11floor} uses an induction argument which relies not only on the fact that the surface $X_d$ is non-singular but also that it is floor decomposed. 
\end{section}

\vspace{2ex}

\textbf{Acknowledgment: }
The author is greatly indebted to Grigory Mikhalkin for his insight and many helpful conversations.  
The author would also like to thank Erwan Brugall\'e, Ilia Itenberg and Johannes Rau for their useful remarks. 
\vspace{0.5cm}

\begin{section}{Preliminaries}\label{sec:prelim}

\subsection{Tropical projective space}
The tropical semi-field is  $\T = \R \cup \{-\infty\}$ equipped with operations
$$``x+y" = \max \{ x, y \} \quad \mbox{and} \quad  ``x \cdot y " = x+y. $$  Remark that the tropical multiplicative identity is $0$ and the additive identity is $-\infty$. 
Thus, $\R^n$ is the tropical tours and $\T^n$ tropical affine space. 
In $\R^n$ we fix the standard lattice $\Z^n$ along with the standard directions $u_i= - e_i$ for $1\leq i \leq n$ and  $u_0 = - \sum_{i = 1}^n u_i.$

 Tropical projective space, as defined by Mikhalkin in \cite{Mik3}, is 
 obtained in
   accordance with classical geometry.     
 It is equipped with tropical homogeneous coordinates
  $$[x_0: \dots : x_N] \sim [x_0 +a : \dots : x_N +a]$$ where $x_i \not = -\infty $ for some $0\leq i \leq N$, and $a \in \R$. Moreover, it is covered by the affine charts 
 $$U_i = \{ [x_0: \dots : x_N] \ | \ x_i = 0 \}  \cong \T^N.$$ Notice also that tropical projective space may be obtained from tropical affine space $\T^n$ by adding a copy of $\TP^{n-1}$ at ``infinity". 
 
 The boundary of $\TP^n$ consists of $n+1$ copies of $\TP^{n-1}$, 
 which we call  boundary hyperplanes. A boundary hyperplane satisfies $H_i = \{ x \in \TP^n \ | \ x_i = -\infty\}$ where  $[x_0: \dots : x_n]$ are the homogenous coordinates. 
The boundary hyperplanes are in correspondence with the $n+1$ standard directions $u_i$. 
Moreover, tropical projective space comes equipped with a natural stratification using the notion of sedentarity in $\T^n$ due to Losev and Mikhalkin \cite{Mik3}. 

\begin{definition}
Let  $x$ be a point in $ \TP^n$ with homogeneous coordinates $[x_0 : \dots : x_n]$ then the sedentarity of $x$ is the set
$$S(x) = \{i \in \{ 0, \dots , n\} \ | \ x_i = -\infty\} \subset \{ 0, \dots n \}.$$
\end{definition}
The above definition does not depend on our representation of $x$ in homogeneous coordinates since adding a  constant in $\R$ cannot change that $x_i = -\infty$. Note that a point has  sedentarity $\emptyset$ if and only if it is contained in the tropical torus $\R^n \subset \T^n$. 

\subsection{Tropical hypersurfaces in $\TP^N$}

\begin{figure}
\begin{center}
\includegraphics[scale=0.25]{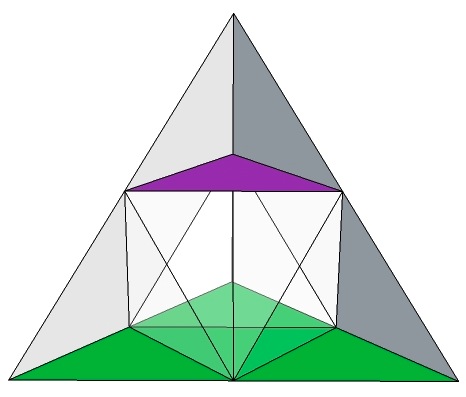}
\hspace{1.5cm}
\includegraphics[scale=0.35]{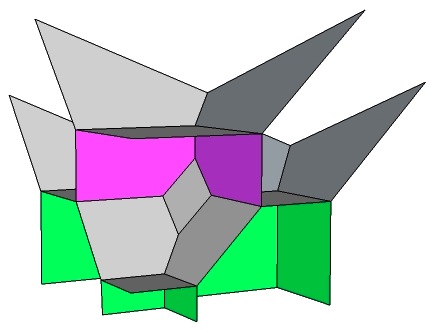}
\end{center}
\caption{A subdivision of the size two simplex and a dual floor decomposed tropical surface of degree 2.}
\label{fig:floorSub}
\end{figure}

Tropical hypersurfaces in $\R^n$ have appeared in various places, \cite{Kap}, \cite{St2}, \cite{Mik3}. Let $\Delta_d \subset \R^n$ be the $n$-simplex of size $d$, for some positive integer $d$.  
For our consideration of surfaces in $\TP^3$ we have $$\Delta_d = Conv \{ (0 ,0, 0), (d, 0, 0),  (0, d, 0), (0, 0, d) \} \subset \R^3.$$
A tropical polynomial is a rational piecewise affine convex function on $\R^n$ given by, 
$$f(x) = ``\sum_{\alpha \in \Delta \cap \Z^n} a_{\alpha}x^{\alpha} \ " = \max_{\alpha \in   \Delta \cap \Z^n}  \{ \alpha + \langle \alpha , x \rangle \}, $$
where $\langle \  , \rangle$ denotes the standard inner product on $\R^n$. 

A tropical polynomial defines a regular subdivision of the lattice polytope $\Delta$, denote it by $\S_f$, see \cite{St2}, \cite{Spe1}, \cite{Mik3}.
The tropical hypersurface $V(f) \subset \R^n$ of such a polynomial $f$ is a polyhedral complex corresponding to the locus of non-differentiability of the function $f$, and equipped with positive integer weights on its facets, again see  \cite{St2},  \cite{Spe1}, \cite{Mik3}.
The tropical hypersurface  $V(f) \subset \R^n$ is dual to the regular subdivision $\S_f$. 
It has degree $d$ if it is dual to a subdivision of $\Delta_d$. 
A subdivision $\S$  of the $n$-simplex $\Delta_d \subset \R^n$ is \textbf{primitive} if each top-dimensional polytope $\Delta^{\prime} \in \S$ has 
$\Vol (\Delta^{\prime}) = \Vol(\Delta_1)$, where $\Delta_1$ is the standard $n$-simplex of size $1$. A tropical hypersurface $V(f) \subset \R^n$ is \textbf{non-singular} if its corresponding dual subdivision $\S_f$ is primitive.

Taking the closure of a tropical hypersurface $V(f) \subset \R^n$ inside tropical projective space gives a non-singular compact tropical surface  $\overline{V(f)} \subset \TP^n$. To summarise we have the following definition. 
 
 \begin{figure}
\begin{center}
\begin{tabular}{|c|c|c|c|}
\hline
& $x \in X^{(0)}$ & $x \in X^{(1)} \backslash X^{(0)}$ & $x \in X^{(2)} \backslash X^{(1)}$ \\ \hline
$s(x) = 0$ & \includegraphics[scale= 2]{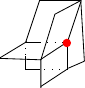} & \includegraphics[scale= 2]{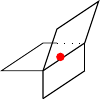} &\includegraphics[scale= 2]{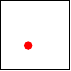}    \\ \hline
$s(x) = 1  $  &  \includegraphics[scale= 2]{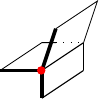} & \includegraphics[scale= 2]{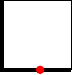} &  \\ \hline
$s(x) = 2$ & \includegraphics[scale= 2]{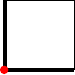} & &  \\ \hline
\end{tabular}
\end{center}
\caption{The neighborhoods of points of a non-singular surface $X \subset \TP^3$ up to integer affine transformation. The bold lines represent points of sedentarity. } 
\label{tab:neigh}
\end{figure}

\begin{definition}
A non-singular tropical hypersurface $X \subset \TP^n$ is the closure of a tropical  hypersurface $V(f) \subset \R^n$ which is dual to a primitive subdivision of the standard  $n$-simplex of size $d$, $\Delta_d \subset \R^n$. 
\end{definition}

When $n = 3$, if a tropical hypersurface is non-singular then the possible neighbourhoods of points in $X$ are shown in Figure \ref{tab:neigh}. 
Notice that from a tropical hypersurface $X \subset \TP^n$ defined by a polynomial $f$ we may uncover the dual subdivsion $\S_f$, for this reason we use interchangeably the  notation $\S(X)$ and $\S_f$ for the dual subdivision. 

\subsection{Floor decomposed surfaces in $\TP^3$.} \label{sec:floor}

Let $\Delta_d \subset \R^3$ be the standard $3$-simplex of size $d$ and consider the affine  planes $Z_k  \subset \R^n$ given by the linear form $\langle x, e_3 \rangle = k$.  The intersection $Z_{d-k} \cap \Delta_d$ for $1 \leq k \leq d$ is 
 a two dimensional  simplex of size $k$, denote it by $\Delta_d(k)$.

\begin{definition}
A non-singular tropical hypersurface $X \subset \TP^n$  is \textbf{floor decomposed} if the dual subdivision $\S(X)$ induces a primitive subdivision on $\Delta_d(k)$  for every $1\leq k \leq d$, (see Figure \ref{fig:floorSub}). 
\end{definition}

The study of floor decomposed tropical hypersurfaces is originally due to Mikhalkin. For $n=2$, floor decomposed curves  have  been used in applications to tropical enumerative geometry, by way of floor diagrams, see for example \cite{Br6b}, \cite{Br7}. 
Their extra combinatorial structure makes them suitable for recursive arguments as we shall see. 

\vspace{0.3cm}
A floor decomposed tropical surface in $\TP^3$ has a very nice decomposition as its name suggests. 
Removing all open faces containing the vertical  direction from a floor decomposed surface $X \subset \TP^3$ there are $d$ connected components, where $d$ is the degree of the surface. These are called the \textbf{floors}. We denote the floor  dual to the part of the subdivision $\mathcal{S}(X)$ laying between hyperplanes $x_3 = d- i$ and 
$x_3=d - i- 1$ by $F_{i+1, i}$.  The floors are two dimensional rational polyhedral complexes however they do not satisfy the balancing condition, which is  familiar  in  tropical geometry, see \cite{St2},  \cite{Spe1}, \cite{Mik3}.  
Two adjacent floors, $F_{i, i-1} $ and $ F_{i+1, i}$,  are joined  by  \textbf{walls}. A wall of $X$ is a connected component of the complement of the floors.   
Figure \ref{fig:walls} shows the floors and walls of the quardric surface  from Figure \ref{fig:floorSub}.
Topologically, the wall  between $F_{i, i-1} $ and $ F_{i+1, i}$ is a  cylinder over a tropical curve $C_i \subset \TP^2$ of degree $i$.
The curve $C_i \subset \TP^2$  is dual to the subdivision $\mathcal{S}(X)$ restricted to the simplex $\Delta_d(k)$ for $k = d-i$ and it is defined by a tropical polynomial $f_i(x_1, x_2)$, obtained  by restricting the polynomial $f(x_1, x_2, x_3)$ defining the surface $X$ to 
the above $2$-simplex and substituting $x_3 = 1_{\T} = 0$.  Denote by $\tilde{C}_i$ the open portion of the surface $X$ corresponding to such a wall.

\begin{figure}
\begin{center}
\hspace{.5cm}
a) \includegraphics[scale=0.35]{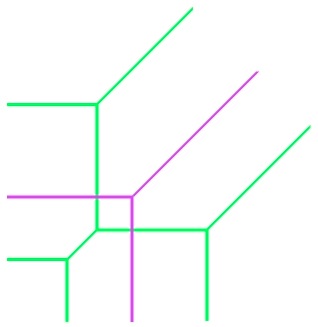}
\hspace{0.1cm}
b)\includegraphics[scale=0.4]{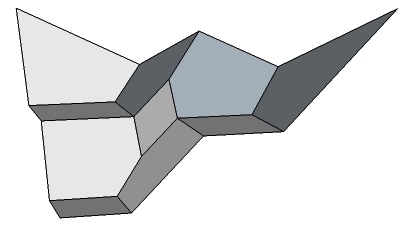}
\vspace{1cm}
\\
c) \includegraphics[scale=0.6]{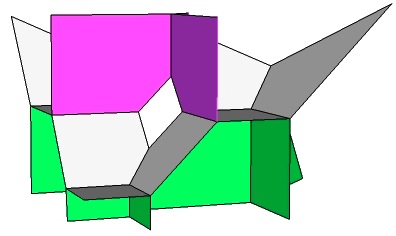}
\end{center}
\caption{a) The floor plan of the degree two floor decomposed surface from Figure \ref{fig:floorSub}. b) The floor corresponding floor $F_{ 2,1}$. c) The surface $X_{2,1}$.}
 \label{fig:walls}
\end{figure}

\begin{definition}
A floor plan for a surface is a collection of tropical plane curves $\{ C_1 , \dots, C_d \}$,  $C_i \subset \TP^2$, such that:
\begin{enumerate}
\item $C_i$ is dual to a primitive subdivision of $\text{Conv}\{(0, 0), (i, 0), (0, i)\}$ for $1 \leq i \leq d$. 
\item for $1 \leq i \leq d-1$,  $C_i$ intersects $C_{i+1}$ in  $i(i+1)$ points contained in the interior of edges of both $C_i$ and $C_{i+1}$.  
\end{enumerate}
\end{definition}

Part a) of Figure \ref{fig:walls}, shows two curves of the floor plan $\{ C_1 , C_2 \} \subset \TP^2$.
A floor decomposed surface $X \subset \TP^3$ determines a floor plan.  Each curve $C_i \subset \TP^2$ is the image of a wall joining floors $F_{i, i-1}$ and $F_{i+1, i}$  under  the linear projection in the vertical  direction, as mentioned above. 
Conversely, given a  floor plan $\{C_1, \dots , C_d\}$, using a pair of curves $C_i , C_{i+1} \subset \TP^2$, we may construct a floor $F_{i+1,i} \subset \TP^3$ of a floor decomposed surface $X$, up to a translation in the vertical direction, in the following way. 
A curve $C_i \subset \TP^2$ is  a tropical hypersurface and so it is given by a tropical polynomial $f_i$. 
The difference $f_{i+1} - f_i$ gives a tropical rational function. 
For a real constant $a \in \R$, the floor
 $F_{i+1, i}$ is simply the graph of $\TP^2$ along the function $a + f_{i+1} - f_i$.
 If the real constants are properly chosen, the graphs corresponding to adjacent floors may be joined via vertical faces and the result is a floor decomposed tropical surface of degree $d$.  Therefore, a floor plan of degree $d$ determines a tropical surface up to the height of the walls joining the adjacent floors.

For  a surface $X_d \subset \TP^3$ of degree $d$, for $i < d-1$, let $X^o_{d,d-1} \subset \TP^3$ be an open neighbourhood of the floor $F_{d, d-1} \subset \TP^3$ and also containing the boundary hyperplane corresponding to $x_3 = -\infty$.  
 Also by removing an open neighbourhood of the bottom floor $F_{d, d-1}$ from $X_d$ we obtain a set  $X_{d-1}^o$. This is an open subset of a  floor decomposed surface $X_{d-1} \subset \TP^3$ with the boundary component $X_{d-1} \cap H_3$ removed. 
Moreover we choose these open sets so that $$X_d = X^o_{d-1} \cup X^o_{d, d-1} \qquad \text{and} \qquad X^o_{d-1} \cap X^o_{d, d-1} = \tilde{C}_{d-1}.$$
Also, $X^o_{d, d-1}$ is an open subset of a tropical surface $X_{d, d-1}$ inside a tropical toric variety $\T(\Delta)$. The $3$-dimensional tropical toric variety is defined by the polytope $$\Delta = Conv \{ (0, 0, 0), (d, 0, 0), (0, d, 0), (d-1, 0, 1), (0, d-1, 1), (0, 0, 1) \}.$$
The surface $X_{d, d-1}$ is the closure in $\T(\Delta)$ of the modification of $\TP^2$ along the function $f_{d} -f_{d-1}$. 
It has $5$ boundary curves, two of which are tropical curves $C_d$, $C_{d-1}$ mentioned above.
See part c) of Figure \ref{fig:walls} for  an example of a surface $X_{2, 1}$.
Both $X_{d-1}$ and $X_{d, d-1}$ are non-singular hypersurfaces of tropical toric varieties. 
 These surfaces  will be used in Section \ref{sec:homo} along with a Mayer-Vietoris sequence to find  $H_{1, 1}(X_d)$  and the intersection form. The computation of   $H_{1, 1}(X_d)$ and its intersection form requires not only that $X_d$ be non-singular but that $X_d$ admits this decomposition.

\end{section}

\begin{section}{Tropical $(p,q)$-homology}\label{sec:tropHomo}

The general definitions of tropical $(p, q)$-homology are given in \cite{ItKaMiZh}.  Here we provide the definitions only in the setting in which we require, namely for non-singular tropical  hypersurfaces in $\TP^n$ and subsets thereof. In fact, we only consider cases for $n \leq 3$.   

\begin{subsection}{The framing groups}\label{sec:framings}
For every point $x \in X$ there is a collection of groups $\F_p(x)$ for $p \in \NN$. These will be the coefficient groups for the $(p, q)$-cells to be defined shortly. 
Define the \textbf{star} of $X$ at $x$ to be the polyhedral complex,  $$Star_x(X) = \{ v \in \R^n \ | \ \exists \epsilon >0, \ s.t. \ \forall 0 < \delta < \epsilon,  x + \delta v \in X\}.$$ This complex contains a single vertex and may be equipped with a fan structure. 
 
 \begin{definition}
Let $X \subset \TP^n$ be a pure dimensional, rational,  polyhedral complex equipped with weight one on the facets and satisfying the balancing condition. 
For  each point $x \in X$ define a collection of  groups $\F_{p}(x)$. Firstly, $\F_0(x) = \Z$, and for $p >0$ define   
\begin{itemize}

\item{If $x \in X$ is a point of sedentarity $\emptyset$, then 
$$\mathcal{F}_p(x) =  < v_1 \wedge \dots \wedge v_p \ | \ v_1, \dots, v_p \in \tau  \subset Star_x(X) \cap \Z^n \ > , $$  
where $\tau $ is any face of $Star_x(X)$.}

\item{If $x \in X$  is a point of sedentarity $I$, $\mathcal{F}_{p}(x)$ is a quotient of the above construction. More precisely, $$\mathcal{F}_p(x) = \frac{< v_1 \wedge \dots \wedge v_k \ | \ v_1, \dots, v_k \in \tau \subset Star_x(X) \cap \Z^n \ > }{ <u_i \ | \ i \in I >}, $$  where $u_{i}$ are the standard directions from Section \ref{sec:prelim}. }
\end{itemize}
\end{definition}

 It is clear that if two points $x, y \in X$ are contained in the same open face of $X$ then $\F_p(x) = \F_p(y)$ for all $p \in \NN$, therefore for each open face $\tau \subset X$ we have groups $\F_p(\tau)$. 

The groups also come equipped with ``inclusion" mappings.  
If $\sigma \subset \overline{\tau}$ are faces of $X$ then there are  maps $i_{\tau, \sigma}: \F_p(\tau) \longrightarrow \F_p(\sigma)$, if $\sigma$ and $\tau$ are both of the same sedentarity then $i_{\tau, \sigma}$ is an inclusion. If the sedentarities of the two faces are not equal, then $i_{\tau, \sigma}$ is a quotient map. 
Denote by $\F_p(X)$ the collection of all of the groups $\F_p(\tau)$ for all open faces $\tau$ of $X$. 
\end{subsection}

\begin{subsection}{Tropical $(p, q)$-homology}
We restrict to recalling the definitions  tropical  homology from \cite{ItKaMiZh} on spaces contained in $\TP^n$. 
Once again, the tropical $(p, q)$-homology group of a space $X$ is essentially the $q$-homology group of $X$ with coefficients in $\F_p(X)$. 

Throughout this section $X \subset \TP^n$ is an open subset of a tropical hypersurface $X_d \subset \TP^n$. All faces of $X$ will be considered to be open unless otherwise stated. A simplicial $q$-cell $a$ is said to respect the polyhedral structure of $X$ if every face of $a$ is contained in a single face of $X$. 
  For an open face $\tau \subset X$, let $C_q(\tau)$ denote the group of singular $q$-chains contained in the closure of $\tau$ with coefficients in $\Z$.  A $q$-cell contained in the face $\tau$ may be equipped with a framing coefficient in $\F_p(\tau)$, this is a $(p, q)$-cell. A $(p, q)$-cell in $\tau$  will sometimes be denoted by $(\phi, a)$ where $\phi \in \F_p(\tau)$ and $a$ a $q$-cell.  A   $(p, q)$-chain in the face $\tau$ is an element of 
$C_{p, q}(\tau) = C_q(\tau) \otimes_{\Z} \F_p(\tau)$. 
Finally, the group of $(p, q)$-chains of $X$ is 
$$C_{p, q}(X) = \bigoplus_{\tau \subset X} C_{p, q}(\tau).$$

The inclusion maps given at the end of the last section  permit the extension of the usual boundary map of $q$-simplicial chains to obtain a boundary map of $(p, q)$-chains. 
This also relies on the fact that the simplicial base of  a $(p, q)$-chain respects the polyhedral structure of $X$. Notice that the boundary map only decreases the $q$-dimension in the chain complex, i.e.~we have $\partial_{p, q}: C_{p, q}(X)  \longrightarrow C_{p, q-1}(X)$. 

As usual, a  $(p, q)$-\textbf{cycle} is an element of the kernel of $\partial_{p, q}$. The $(p, q)$-cycles of $X$ will be denoted $Z_{p, q}(X)$.  A $(p, q)$-\textbf{boundary} is an element of the image of $\partial_{p, q+1}$ 
and the homology groups are: $$H_{p, q}(X) = \frac{Ker(\partial_{p, q})}{Im(\partial_{p, q+1})}.$$ If two $(p, q)$-cycles $\alpha$, $\alpha^{\prime}$ are equivalent homology classes  write $\alpha \sim \alpha^{\prime}$.

\begin{example}[Tropical homology of a curve in $\TP^2$]\label{ex:homocurve}
This example calculates the $(p, q)$-homology groups of a non-singular tropical curve $C \subset \TP^2$. This will also be applied when  finding the tropical $(1, 1)$-homology of a floor decomposed surface in Section \ref{sec:H11}.

A non-singular tropical curve $C \subset \TP^2$ is dual to a primitive subdivision of the simplex $$\Delta_d = \{ (0, 0), (d, 0), (0, d) \}.$$ The curve $C$  is a  graph with first Betti number $g = (d-1)(d-2)/2$.  Because the curve is one dimensional all $(p, q)$ groups besides $H_{1, 1}, H_{1, 0}, H_{0, 1}$ and $H_{0, 0}$ are zero. If $p = 0$ then $\F_p(x) = \Z$ therefore $H_{0, 0}(C) = \Z$ and $H_{0, 1}(C) = \Z^g$.

The group $H_{1, 0}(C)$ consists of classes of points $x$ equipped with a framing vector $\F_1(x)$. 
At a leaf $x$ ($1$-valent vertex) of $C$ the framing group $\F_1(x)$ is  $0$. If $C$ is a tree then for a $1$-framed point $(\phi, x)$ on $C$ we may find a $(1, 1)$-chain with boundary supported on $x$ and the leaves of $C$, so  $x \sim 0$. 
If $C$ is not a tree, 
let $x_1, \dots ,  x_g$  be a collection of framed breaking points of $C$. Recall that a breaking point means that each $x_i$ is in the interior of an edge of $C$ and $C \backslash \{x_1, \dots , x_g\}$ is a connected tree. From $C \backslash \{x_1, \dots , x_g\}$, construct a tree $C^{\prime}$ by adding  vertices each of the $2g$ open edges obtained after removing the breaking points. 
Moreover, declare $\F_1$ at each of these new vertices to be  $0$.
The framing of a point $x_i$ is a choice of primitive integer vector parallel to the edge of $C$ containing the point.  Now since $C^{\prime}$ is a tree, for any $1$-framed point we may find a $(1, 1)$-chain whose boundary as a singular $1$-chain is supported on that point and  the leaves of $C^{\prime}$. Therefore, any $1$-framed point is equivalent to a $(1 ,0)$-cycle supported on the points $x_1, \dots , x_g$ equipped with some framings. 
In fact these form a basis of $H_{1,0}(C)$. To see this, suppose there is a  $(1, 1)$-chain $\tau$ in $C$ bounding a linear combination of the framed points $x_1, \dots , x_g$. Then $\tau$ gives a  $(1, 1)$-cycle on the tree $C^{\prime}$ since  $\F_1(x) = 0$ for all leaves, including those corresponding to breaking points. As a cycle in $C^{\prime}$ denote $\tau$ by the same name. 
Since $C$ is non-singular, if $\tau$ contains an edge adjacent to a vertex $v$, then $\tau$ must also contain the other two edges adjacent to $v$, because every non-leaf vertex is trivalent, and $\partial \tau = 0$. Moreover the framings of the other two edges adjacent to the vertex are determined. Therefore, since $C^{\prime}$ is connected, the support of $\tau$ must be the entire curve and $\tau$ must be a multiple of the ``parallel cycle" $C^{\prime}$, see Section \ref{sec:cycmap}.  Now $\tau$ considered as a $(1, 1)$-chain on $C$ still has empty boundary. Therefore, the framed points $x_1, \dots , x_g$ form a basis of $H_{1, 0}(C)$. 
We remark that if $C$ is a singular curve, the above framed points would not necessarily be independent, since the curve  $C^{\prime}$ could have other $(1, 1)$-cycles.  

There is a way to construct paired bases for  $H_{1, 0}(C)$ and $H_{0, 1}(C)$. 
Starting from a collection of framed breaking points $x_1, \dots, x_g$ we can construct a  basis of $H_{0, 1}(C)$. Adding back to $C^{\prime}$ a single breaking point $x_i$ we obtain a curve with first Betti number $1$. For $1 \leq i \leq g$, let $\gamma_i$ be the unique embedded loop oriented coherently with the framing of $x_i$. This gives a basis of $H_{0,1}(C)$ dual to the basis of $H_{1, 0}(C)$ determined by the breaking points. 
This descends to  an intersection pairing on the homology groups
$$\langle \ \  , \  \ \rangle : H_{1, 0}(C) \times H_{0, 1}(C) \longrightarrow \Z.$$
Here we will not use this pairing but only the bases construction above. 

Finally, the group $H_{1, 1}(C)$ is generated by the $(1,1)$-cycle that is supported on  the entire curve $C$. Choose any orientation on the edges of $C$ and then equip each edge $e$ of $C$ with the framing given by  the primitive integer vector in the direction of $e$ and whose orientation is coherent with the choice of orientation of the edge $e$. Once again, this is because the curve $C$ is non-singular. 
\end{example}

\end{subsection}

\begin{subsection}{Intersection of $(1, 1)$-cycles}
The idea of intersecting tropical $(p, q)$-cycles under appropriate conditions is due to Grigory Mikhalkin. In 
 \cite{KMS} it is shown that $(1, 1)$-classes may be intersected on any non-singular compact tropical  surface when $H_{1, 1}(X)$ is torsion free. 
Therefore there is 
a symmetric bilinear pairing on $(1, 1)$-homology classes $$. : H_{1, 1}(X) \times H_{1, 1}(X) \longrightarrow \Z.$$ 
The next definitions recall the case of transversally  intersecting cycles, referring the reader to \cite{KMS} for the proofs that the intersection product is well-defined on $H_{1, 1}$.  

\begin{definition}
Let $X$ be a non-singular tropical surface,  two $(1, 1)$-cycles   $\alpha, \beta$ in $X$ intersect transversally in the underlying $1$-chains meet transversally in interior of $1$-cells of $\alpha$ and $\beta$. In addition all points of intersection must be contained in the interior of top dimensional faces of $X$.
\end{definition}

When two $(1, 1)$-cycles intersect transversally in $X$ the following defines their intersection multiplicity. 
\begin{definition}\label{transInt}
Let $\alpha, \beta$ be two transversally intersecting $(1, 1)$-cycles in $X$.
\begin{enumerate}

\item Suppose the point  $x \in \alpha \cap \beta$ is contained in the facet $F \subset X$ and also the interior of the framed one cells $(\phi_{\alpha}, a)$ and $(\phi_{\beta}, b)$ where $a$ and $b$ are primitive chains. Then the intersection multiplicity of $\alpha$ and $\beta$ at $x$ is:
$$m_x = (-1)^{\delta} [\Lambda_F: \Lambda_{\phi_{\alpha}} \oplus \Lambda_{\phi_{\beta}}],  $$ where $\delta =0$ if the orientation of the face $F$ induced by $a, b$ is the same as that induced by $\phi_a, \phi_b$ and $\delta = 1$ if the orientations are opposite.  

\item The intersection product of the two cycles is, 
$$\alpha . \beta = \sum_{x \in \alpha \cap \beta} m_x.$$

\end{enumerate}
\end{definition}

\begin{figure}
\begin{center}

\includegraphics[scale=0.35]{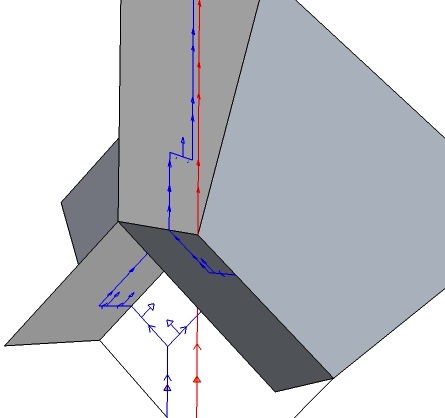} 
\put(-15, 25){$X$}
\put(-110, 80){$\beta$}
\put(-83, 115){$\alpha$}
\hspace{0.5cm}
\includegraphics[scale=0.18]{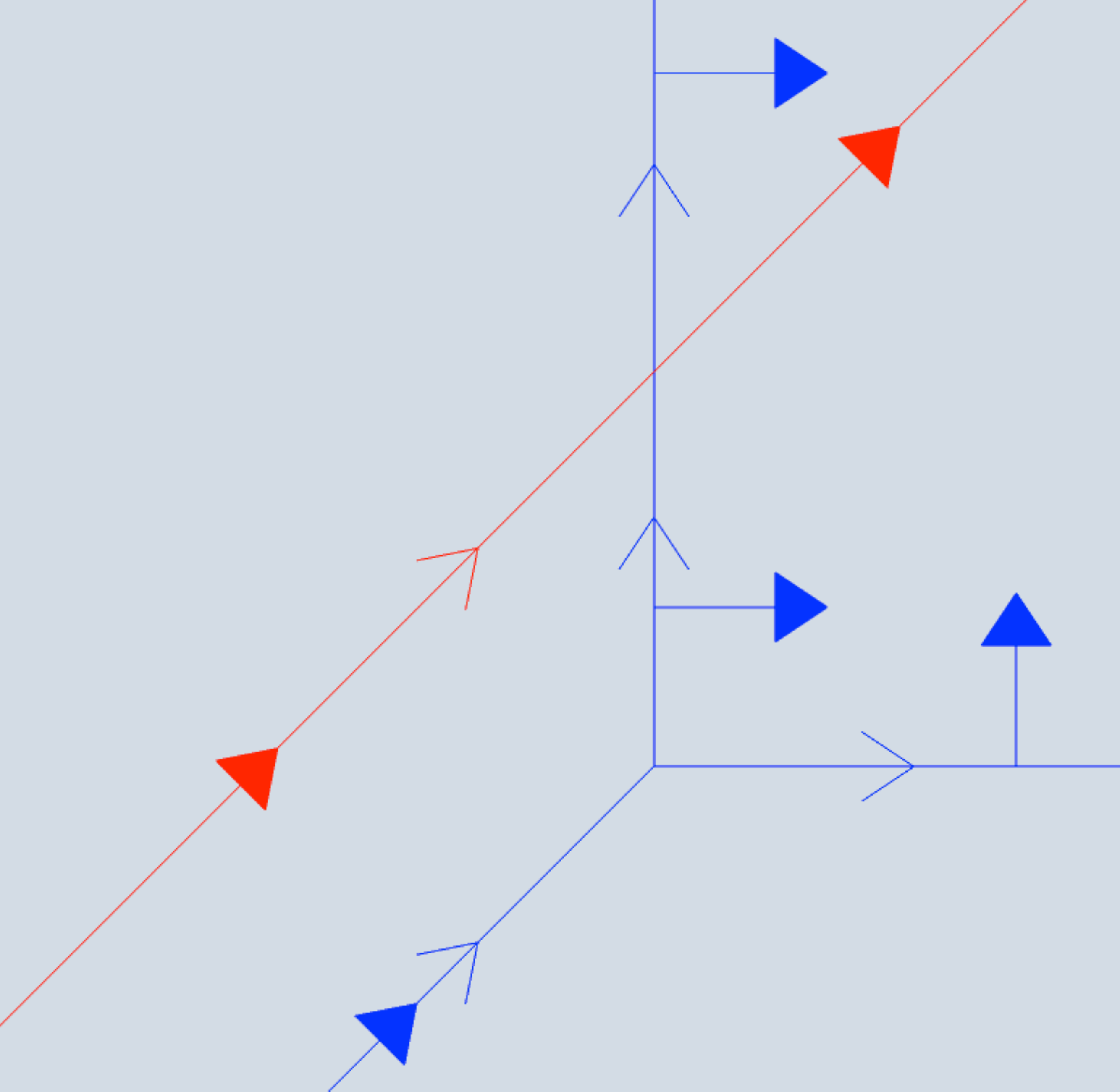}
\put(-55, 85){$x$}
\put(-70, 20){$\beta$}
\put(-93, 45){$\alpha$}

\end{center}
\caption{a) Two $(1, 1)$ homology cycles in a plane. b) A close up of their only point of intersection. Solid arrows denote their framings and the other arrows the orientations of the cells. \label{11Int}} 
 \label{fig:-1int}
\end{figure}
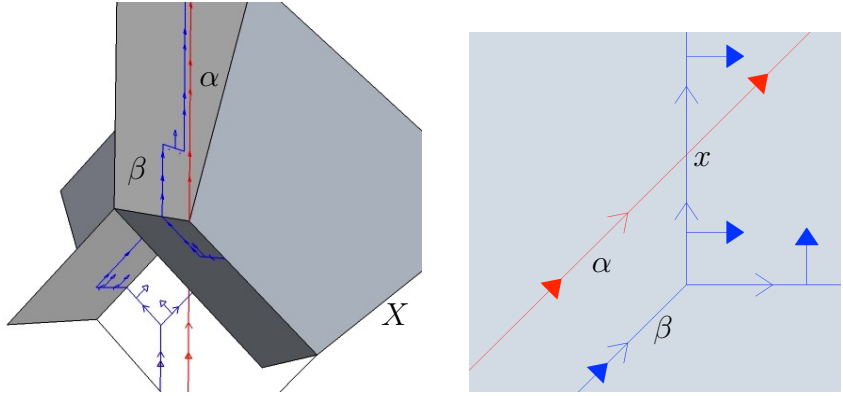

As previously mentioned, it is shown in \cite{KMS} that for a compact non-singular tropical surface $X$, we may always find representatives of $(1, 1)$-classes intersecting transversally and the intersection product defined above does not depend on our choice of representatives in the case when $H_{1, 1}(X)$ is torsion free. 

\begin{example}\label{negInt}
Consider the two  $(1, 1)$-cycles shown in a portion of a tropical surface $X$ in Figure \ref{fig:-1int}.  The cycle $\alpha$ is supported on an affine line in the direction $(0, 0, 1)$ and has framing parallel to and oriented in the same way as the underlying simplicial chain. The framings of the cycle  $\beta$ are indicated with solid arrows and the other arrows indicate the  orientation of the underlying chains. The two cycles intersect transversally at the point $x$, which is  in the interior of the facet of the surface generated by the vectors $(-1, 0, -1)$ and $(0, -1, -1)$. At the point of intersection the framing vectors of $\alpha$ and $\beta$ are $(0, 1, 1)$ and $(0, 0, 1)$ respectively. Taking the orientations of the cycles into account, by Definition \ref{transInt} the intersection multiplicity of $\alpha$ and $\beta$ at $x$ is then $-1$. 
In Section \ref{sec:form} the above cycles will reappear in the compact surface $X_{i+1, i}$. Here we have $\alpha \sim \beta$ and their only point of intersection is shown in Figure \ref{fig:-1int}, so that $\alpha^2 = \alpha . \beta = -1$. 
\end{example}

If a  $(1, 1)$-cycle is supported on a collection of rational sloped  $1$-chains and the framing vector  of each edge is a coherently oriented integer vector in the direction of the edge, then we say that the $(1, 1)$-cycle is a \textbf{parallel cycle}. The cycle $\alpha$ from Example \ref{negInt} is a parallel cycle. 
Two parallel $(1, 1)$-cycles intersecting transversally in the interior of a facet of a surface always have a positive intersection multiplicity.

\end{subsection}

\begin{subsection}{Tropical $1$-cycles and the cycle map}\label{sec:cycmap}
Tropical $k$-cycles predate and are different from $(p, q)$-cycles \cite{Mik3}, \cite{St2}.  They are meant to be analogous to algebraic, rather than topological cycles. 
From the authors of \cite{ItKaMiZh} there is a cycle map which represents a  tropical $k$-cycle by a parallel $(k, k)$-cycle. 
This map is analogous to the cycle map in classical algebraic geometry, which produces from a algebraic $k$-cycle a class in $H^{k, k}(\X) \subset H^{2k}(\X)$, for an algebraic variety $\X$ with sufficiently nice properties, see Chapter 19 of \cite{FulInt}. We recall the definition of tropical $1$-cycles and explain the cycle map from \cite{ItKaMiZh} in dimension $1$.

\begin{definition}\label{1cyc}
A tropical $1$-cycle $A \subset \R^n$ is a one dimensional rational polyhedral complex equipped with integer weights  on its edges and satisfying to the balancing condition at every vertex $v \in A$:
$$ \sum_{v \in e} w_{e}v_e = 0 ,$$
where $w_e$ is the weight of an edge $e$ and $v_e$ is the primitive integer vector in the direction of $e$ pointing outward from the vertex $v$. 

A sedentarity order $0$ tropical $1$-cycle $A \subset \TP^n$ is the closure of a tropical $1$-cycle $A^o$ in $\R^n \subset \TP^n$. 
\end{definition}

A tropical $1$-cycle $A \subset \TP^n$ yields a parallel  $(1, 1)$-cycle as defined after Example \ref{negInt} via the cycle map $$Cyc: Z_1(X) \longrightarrow Z_{1, 1}(X).$$ This map is also due to the authors of  \cite{ItKaMiZh}. 
To obtain the $(1, 1)$-cycle, $Cyc(A)$, first orient every edge of $A$, this is the $1$-chain supporting $Cyc(A)$. Then the framing of an oriented edge $e$ is the primitive integer vector parallel to the edge and coherent with the previously chosen orientation of that edge. Finally, each $(1, 1)$-cell from an edge is multiplied by the integer weight  of the corresponding edge in the $1$-cycle $A$.  It is easy to see that the resulting  $(1, 1)$-chain is closed due to the balancing condition at the vertices given in Definition \ref{1cyc}.

Given a pair of  tropical $1$-cycles in a non-singular surface $X \subset \TP^3$ which are not contained in the boundary their intersection product is a well defined $0$-cycle
  supported on the points in $(A \cap B)^{(0)}$, and denoted $A.B$. Once again, the points of $A.B$ come equipped with integer multiplicities which are  determined only by local data and are fully described in \cite{MichPollo} as well as earlier in \cite{Shaw} in the sedentarity zero case.  It is shown in \cite{KMS} that the sum of all multiplicities of points in  $A.B$ is equal to $Cyc(A).Cyc(B) \in \Z$. Therefore, the total intersection multiplicity of two $1$-cycles may be calculated by intersecting the $(1,1)$-cycles $Cyc(A)$ and $Cyc(B)$, justifying the use of the same intersection symbol for both products.

\end{subsection}

\end{section}

\begin{section}{Tropical $(1, 1)$-homology for floor decomposed surfaces} \label{sec:homo}

This section computes the tropical $(1, 1)$-homology groups of a floor decomposed surface  $X_d$.
This is done first by explicitly constructing a collection of cycles on $X_d$. Then a  Mayer-Vietoris type argument is used to prove that these cycles form a basis of $H_{1, 1}(X_d)$. 

\subsection{Cycles on a floor decomposed surface}\label{sec:basis}

\begin{figure}
\begin{center}
\includegraphics[scale=0.3]{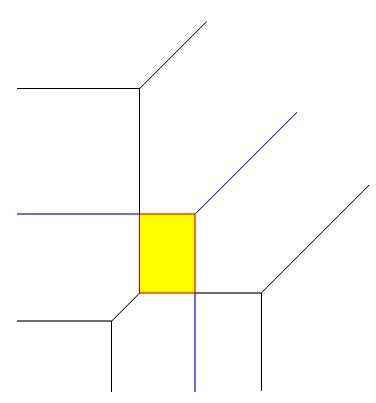}
\hspace{2cm}
\includegraphics[scale=0.6]{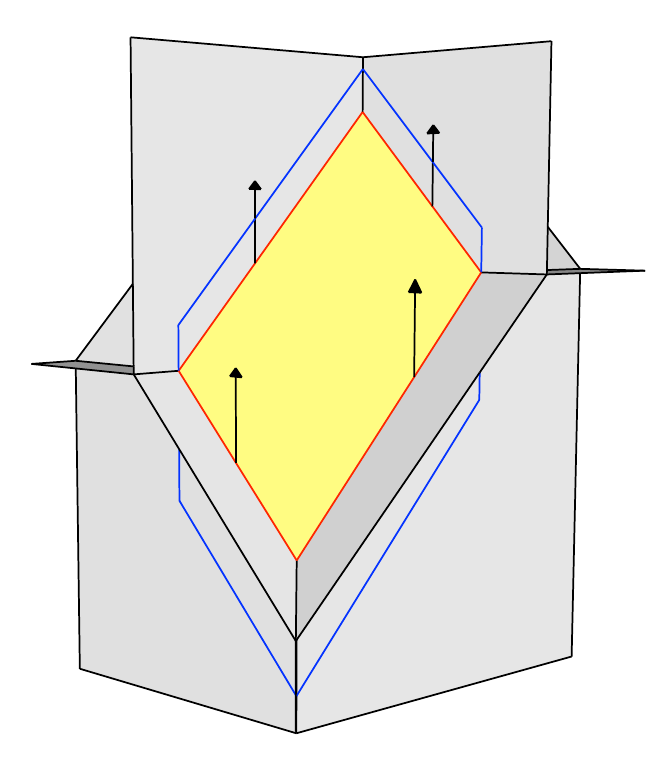}
\end{center}
\caption{ A $(1, 1)$-floor cycle on a surface $X_{2, 1}$}
\label{fig:floorcyc}
\end{figure}

\begin{definition}\label{def:floorcycle}
A \textbf{floor cycle} on $X_d \subset \TP^3$ (respectively $X_{ d, d-1}^o, X_{d-1}^o \subset 
  \T(\Delta)$) is any  simplicial $1$-cycle supported on the $1$-skeleton of a floor $F_{ i+1, i}^{(1)}$ and equipped with a constant vertical framing. 
\end{definition}

The right hand side of Figure \ref{fig:floorcyc} shows a floor cycle on the surface $X_{2, 1}$. The right hand side of the same figure shows the cycle supported on the floor plan.
Notice that without the framing these cycles are homologous to zero in $H_1(X_{d}, \Z)$.  
Next we describe a collection of floor cycles which will form part of the basis of $H_{1,1}(X_d)$. 

\subsubsection{Cycles on floors} \label{subsubfloors}
On every floor $F_{i+1, i}$ there are $i(i+1)$ points in the 0-skeleton $F_{i+1, i}^{(0)}$ which arise as intersection points of the curves $C_i, C_{i+1}$ in the floor plan of $X_d$. Label these points $x_1, \dots,  x_{i(i+1)} \in F_{i+1, i}^{(0)}$. For any pair of points $x_s, x_t$ there is a path in $F_{i+1, i}^{(1)}$ connecting them contained entirely in $C_i$  and similarly for $C_{i+1}$. These paths can be concatenated and equipped with an orientation and a vertical framing to produce a $(1, 1)$-floor cycle on $F_{i+1, i}^{(1)}$.  Figure \ref{fig:floorcyc} shows a floor cycle in $X_{2, 1}$. 
For $1 \leq s < i(i+1)$ consider the pairs of points  $(x_s, x_{s+1})$ on the floor  $F_{i+1, i}^{(0)}$. By the above construction these yield $i(i+1) -1$  floor cycles for each floor.  Denote this collection of $(1, 1)$-cycles by $A$. There are exactly 
$$\sum_{i=1}^{d-1} i^2 + i -1 = \frac{d^3 - 4d + 3}{3}$$
such cycles on the surface $X_d$ for $d > 1$. For $d = 1$ there are no such cycles.

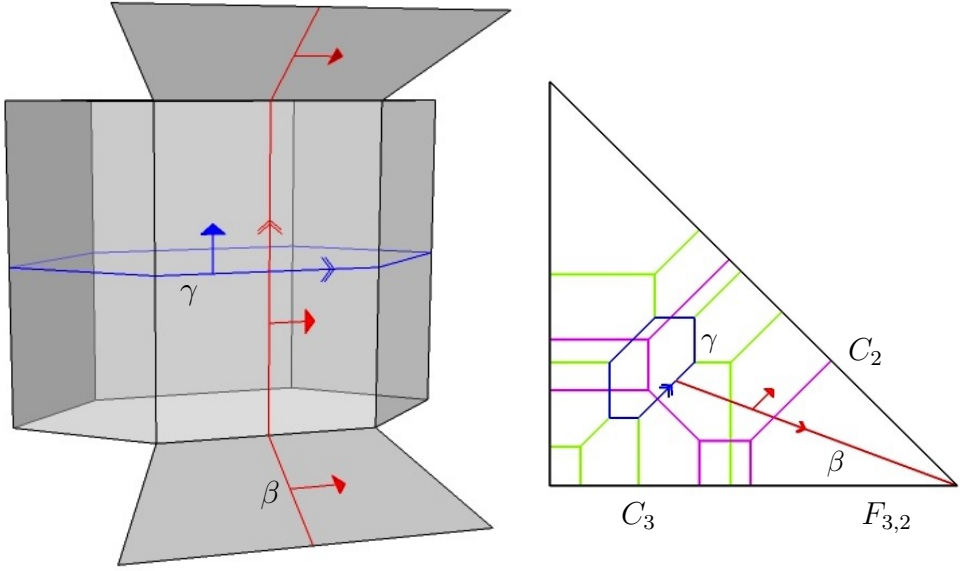
\begin{figure}\label{fig:betagamma}
\includegraphics[scale=0.5]{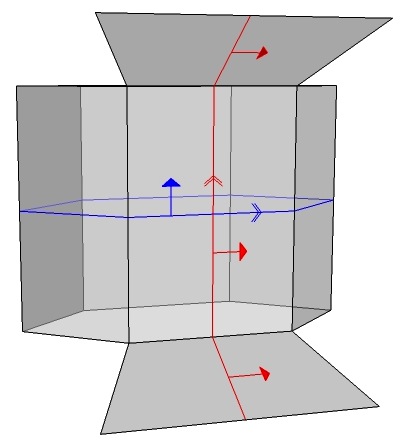}
\includegraphics[scale=0.4]{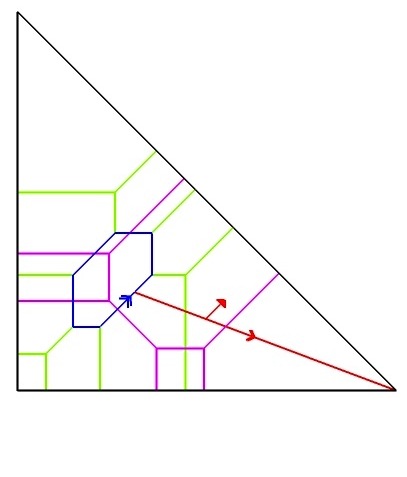}
\put(-40, 20){$F_{3, 2}$}
\put(-295, 105){$\gamma$}
\put(-265, 28){$\beta$}
\put(-53, 40){\small{$\beta$}}
\put(-100, 87){\small{$\gamma$}}
\put(-45, 82){$C_2$}
\put(-130, 20){$C_3$}
\caption{The left shows the cycles $\beta$, $\gamma$ in a part of a quartic surface $X_4$. On the right is the view of the cycles on the floor $F_{3, 2} \subset X_4$. \label{fig:alphabeta}}
\end{figure}

\subsubsection{Cycles between floors}\label{subsubwalls} 
The next collection of cycles are also floor cycles but which are free to move between two floors. Two adjacent floors $F_{i+1, i}, F_{i, i-1}$ are connected by a wall corresponding to the curve $C_i$.  A simplicial $1$-cycle with $\Z$ coefficients is by definition also a $(0, 1)$-cycle. A $(0, 1)$-cycle on $C_i$ yields two $(1, 1)$-floor cycles supported on  $F_{i+1, i}$ and $F_{i, i-1}$ by equipping it  with a vertical framing vector just as was done above. However, the two $(1, 1)$-cycles are homologous in $X_d$. To see this, first take a simplicial $2$-chain bounding the difference of the two underlying simplicial $1$-cycles. Equipping this $2$-chain with constant vertical framing, we obtain a $(1, 2)$-chain which bounds the two $(1, 1)$-floor cycles. Since we are interested in a basis we only take one of these two cycles.  
In Figure \ref{fig:alphabeta}, the blue cycle $\gamma$ is contained on the wall corresponding to $C_3$ and is homologous to the floor cycles just described which come from a cycle  on $C_3$. 
Let $C$ denote the set of floor cycles arising from all of the curves $C_i$ and contained in $F_{i+1, i}$ for  $1 \leq i <d$. 
In total on the surface $X_d$ there are 
$$ \frac{d^3 -6d^2 +11d - 6}{6}$$
such cycles. This is also exactly the second Betti number of $X_d$, $b_2(X_d)$

\subsubsection{Cycles joining floors}\label{subsubjoin}
Since $C_i$ is a non-singular tropical curve, we may choose a collection of $(1, 0)$-cycles on $C_i$ which represent classes dual to the basis of cycles chosen above for the $(0, 1)$-cycles to give the cycles in $C$, see Example \ref{ex:homocurve}. A $1$-framed point in $C_i$ gives a $(1, 1)$-cycle on $X_d$ in the following way:
A $1$-framed point $x \in C_{i}$ can be lifted to two framed points  $y \in F_{i+1, 1}$ and $y^{\prime} \in F_{i, i-1}$ both contained on edges of the floors.  Moreover, $y^{\prime} - y$ bounds a  $(1, 1)$-cell on the wall joining $F_{i+1, i}$ and $F_{i, i-1}$.  Now the framed point $x \in C_{i}$ is a boundary in $\TP^2$ of a $(1, 1)$-cell $\tau$. This cell $\tau$ has framing parallel to the framing of $x$ and endpoints on $x$ and any point of sedentarity $2$ of $\TP^2$, see Figure \ref{fig:alphabeta}.  The  projection $\pi: F_{i+1, i} \longrightarrow \TP^2$  is one-to-one so we may lift this $1$-cell in $\TP^2$ to a $1$-chain in $F_{i+1,i}$. Moreover, each cell in this chain can be equipped with a unique $1$-framing so that under the projection to $\TP^2$ it is the framing of the $(1, 1)$-cell bounding $x$. 
Now the boundary of this $1$-chain is the framed point $y$ and a collection of points on the $1$-skeleton of the floor equipped with the vertical framing. The framing of these points must be vertical since the projection $\pi$ applied to the framing of any cell is constant, it is the framing of $\tau$. These vertically framed points in $X_d$ are also boundaries. In fact, there is a bounding chain which lives in the portion of $X_d$ lying entirely above the floor $F_{i+1, i}$. This bounding chain consists of vertically framed $1$-chains supported on the walls of the surface $X_d$. These chains are connected at adjacent walls via the vertices of the surface.  Finally at the top floor $F_{1, 0}$ the vertically framed chain can be connected to three parallel framed rays in the directions $-e_1, -e_2, e_0$. This $(1, 1)$-chain then has boundary supported on $y$ and point of sedentarity matching their framings. Therefore the boundary is equal to the framed point on the floor $F_{i+1, i}$. 

Similarly, vertically framed points on $F_{i, i-1}$ are boundaries of vertically framed $(1, 1)$-chains supported on the walls of $X_d$ lying strictly below the floor $F_{i, i-1}$ and joined at vertices of $X_d$. The boundary of such a $(1, 1)$-chain is supported on the point on the floor $F_{i, i-1}$ as well as points on the boundary $z = -\infty$. Then the boundary as a $(1, 0)$-cycle is simply the framed point on the floor. The cycle $\beta$ in Figure \ref{fig:alphabeta} depicts the portion of such a cycle contained in the wall corresponding to $C_i$. 
As a part of the basis we take one such cycle for each class in $H_{1, 0}(C_i)$ for all of the curves $C_i$ in the floor plan of $X_d$. Therefore, we take 
$$b_2(X_d) =  \frac{d^3 -6d^2 +11d - 6}{6}$$
such cycles in the basis. 

These cycles come in pairs with cycles between the walls constructed above in \ref{subsubwalls} arising from classes in $H_{1, 0}(C_i)$. 
The pairs coming from dual classes in $H_{1, 0}(C_i)$ and $H_{0, 1}(C_i)$ will be denoted by $\beta$ and $\gamma$ respectively as in Figure \ref{fig:alphabeta}.

\comment{
$F_{d, d-1}$, and $y^{\prime}_i$ is a boundary in $F_{d-1, d-2}$. To see this, consider a path from $x_i \in C_{d-1} \subset \TP^2$ to a corner, as in the righthand side of  Figure \ref{fig:alphabeta}. Lifting this path to both $F_{d-1, d-2}$ and $F_{d, d-1}$ it bounds a collection of vertically framed points which are all homologous to zero in $X$. This produces a cycle $\beta_i$. 
The cycles joining the floors are very much}

\subsubsection{The vertical cycle} 
The last $(1, 1)$-cycle in the basis is supported on the walls of $X_d$ and the top floor $F_{1, 0}$. The $1$-chains supported on the walls are vertically framed and join adjacent walls via the vertices, as above in \ref{subsubjoin}.
On the top floor the $(1, 1)$-cycle is again a union of three parallely framed rays  in the directions $-e_1, -e_2,$ and $e_0$ also as above. 

\vspace{0.5cm}

In the next section we will prove that together the above cycles give a basis for $H_{1, 1}(X_d)$ by using a Mayer-Vietoris sequence and induction on the degree of the surface.

\subsection{Mayer-Vietoris sequence}\label{sec:H11}
For $X$  a tropical manifold in the sense of \cite{KMS} (or even more generally a tropical variety) the tropical homology groups satisfy a Mayer-Vietoris sequence similar to homology with constant coefficients. 
Here the sequence will only be applied to non-singular floor decomposed tropical surfaces $X \subset \TP^3$. 

\begin{proposition}
Let $A, B \subset X$ be open subsets such that $X = A \cup B$, then the following sequence of tropical homology groups is exact:
\begin{equation}\label{eq:MV} \cdots \longrightarrow H_{p, q+1}(X)  \overlim{\partial_{\ast}} H_{p,q}(A \cap B)  \overlim{i_{\ast}}  H_{p, q}(A) \oplus H_{p,q}(B)  \overlim{j_{\ast} - k_{\ast}} H_{p, q}(X)  \overlim{\partial_{\ast}} \cdots 
\end{equation}
\end{proposition}

The maps above are as usual in the Mayer-Vietoris sequence. The  map $\partial$ comes from expressing a representative  $\tau$ of a class in $H_{p, q}(X)$
as a sum of two $(1, 1)$-chains $\alpha, \beta$, contained in $X_{d-1}^o$ and $X_{d, d-1}^o$ respectively, then $\partial_{\ast} \tau =  \partial \alpha  \in 
H_{p, q-1}(A \cap B)$.
As usual this is well-defined on homology classes.   
The maps
 $j_{\ast}$, $k_{\ast}$ are given by the respective inclusions of chains 
$$j: C_{p, q}(A)  \longrightarrow C_{p, q}(X), \qquad k: C_{p,q}(B) \longrightarrow C_{p, q}(X).$$

\begin{proof}
The proof follows the same way as in the case of constant coefficients,  see \cite{hAT}. 
First we prove that the following diagram is commutative and exact:
$$\xymatrix{ 0  \ar[r] &  C_{p, q}(A \cap B)   \ar[d]_{\partial}    \ar[r] & C_{p, q}(A) \oplus C_{p,q}(B)\ar[d]_{\partial}  \ar[r] & C_{p, q}(X) \ar[d]_{\partial} \ar[r] & 0  \\ 
0 \ar[r] & C_{p, q-1}(A \cap B)  \ar[r] &   C_{p, q-1}(A) \oplus C_{p,q-1}(B)   \ar[r] & C_{p, q-1}(X) \ar[r] & 0}$$
First we verify exactness in the short exact sequences and then commutativity.  If a  chain from $C_{p, q}(\tilde{C}_{d-1})$  is zero in both $C_{p, q}(X_{d-1}^o)$ and $C_{p,q}(X_{d, d-1}^o)$ it must be the zero chain.  Also it is clear by definition of the maps that $Im( i ) \subseteq Ker(j-k )$. If $(\tau_1, \tau_2) \in Ker(j-k)$, then $\tau_1 = \tau_2$ and so both $\tau_1, \tau_2$ must be chains in $C_{p, q}(\tilde{C}_{d-1})$ and $(\tau_1, \tau_2) = (\tau_1, \tau_1) = Im(i)$. Lastly, surjectivity at $C_{p, q}(X_d)$ follows  since any chain in $X_d$ may be written as a sum of two chains, one in $X^o_{d-1}, X^o_{d, d-1}$. 

The boundary operator commutes with the inclusion maps $i, j, k$, therefore the entire diagram is commutative. 
Now the long exact sequence follows from a standard diagram chase in homological algebra exactly the same as for the Mayer-Vietoris sequence for homology with constant coefficients, see \cite{hAT}. 
\end{proof}

For $X_d$ a non-singular floor decomposed tropical  surface  of degree $d$, recall the definitions of the associated polyhedral complexes $X_{d-1}^o$,  $X_{d, d-1}^o$ and $\tilde{C}_{d-1}$ from Section \ref{sec:floor}. 

\begin{lemma}\label{lem:othergroups}
For a non-singular floor decomposed surface $X_d$, we have
\begin{enumerate}[i)]
\item  $h_{1, 0}(X_{d-1}^o) = h_{1,0}(X_{d, d-1}^o) = 0$
\item $h_{1,0}(\tilde{C}_{d-1}) = g(C_{d-1})+1$ 
\item  $h_{1,1}(\tilde{C}_{d-1}) = g(C_{d-1})$
\item $h_{1, 2}(X_{d-1, d}) = 0$
\end{enumerate}
\end{lemma}

\begin{proof}
For the first statement, if a $1$-framed point is equipped with vertical framing $\pm e_3$ then the point $x$ on which it is supported must be contained in a face of 
$ X_{d, d-1}^o$ which either contains the vertical direction or is adjacent to a face of $X_{d, d-1}^o$ which does. Take as base of a $(1, 1)$-chain any path which joins the point $x$ to a point on the boundary hyperplane $z_3 = -\infty$, and which is contained only in the walls and vertices of $X_{d, d-1}^o$. Then this simplicial $1$-chain can be equipped with a vertical framing to obtain a $(1, 1)$-cycle. Its boundary is the vertically framed $(1,0)$-chain supported on $x$, so this $(1, 0)$-chain is homologous to zero.
 Therefore we may suppose the framing $v \in \Z^3$ of a $(1, 0)$-cycle is orthogonal to $e_3$, in other words $\langle v, e_3 \rangle =0$, where the brackets denote the standard inner product on $\R^3$.
Let $\pi: X_{ d,d-1}^o \longrightarrow \TP^2$ be the extension of  the projection in the vertical direction in an affine chart . The framed point $(\pi(\phi), \pi(p))$ is homologous to zero in $\TP^2$. We may lift the $(1, 1)$-chain which has $(\pi(\phi), \pi(p))$ as its boundary back to $X_{d, d-1}^o$ by taking the graph of the rational function $f_{d}-f_{d-1}$ from which $X_{d, d-1}^o$ is constructed to a $(1, 1)$-chain $\tau$ in $X^o_{d, d-1}$. 
Then the boundary is:
$$\partial \tau = ( \phi, p) + \sum_{i = 1}^k (\pm e_3, p_i),$$ where the points $p_i$ lie on the $1$-skeleton of $X_{d, d-1}^o$. All vertically framed points are homologous to $0$ by the above argument, and we have $(p, \phi) = \partial \tau^{\prime}$, and so it is also homologous to $0$.  A similar argument shows that $H_{1, 0}(X_{d-1}^o) = 0$. 

For the second statement,  first notice that a point with vertical framing $\pm e_3$ on $\tilde{C}_{d-1}$ represents a non-zero class in $H_{1, 0}(\tilde{C}_{d-1})$. 
If the framing of a point in $\tilde{C}_{d-1}$ is orthogonal to $e_3$, then it also represents a framed point on the underlying curve $C_{d-1}$. These cycles are independent in $\tilde{C}_{d-1}$, since they are independent in $C_{d-1}$.     From Example \ref{ex:homocurve} it follows that, 
$h_{1, 0}(\tilde{C}_{d-1}) = g(C_{d-1})+1$.  

For $H_{1, 1}(\tilde{C}_{d-1})$, we may assume every $(1, 1)$-class has a representative which is bounded. This is because there are only $3$ directions in $\Z^3$ which become zero in $\F_1$ on the boundary of $\tilde{C}_{d-1}$ and these $3$ directions are not balanced, namely they are $(-1, 0, 0), (0, -1, 0)$ and $(1, 1, 1)$. It also follows that for each cycle we may choose a representative which has constant vertical framing, so that the support of the $(1, 1)$-chain is also a cycle in homology with $\Z$ coefficients. 
Take $g(C_{d-1})$ vertically framed bounded $(1, 1)$-cycles corresponding to a basis of $H_1(\tilde{C}_{d-1})$.  
The claim is that these cycles remain independent in $\tilde{C}_{d-1}$. If a $(1,2)$-chain $\tau$ bounds some combination of these vertically framed $(1, 1)$-cycles $\sigma$, then as a simplicial $1$-chain the boundary of $\tau$ must be supported on the vertically framed cycles and the boundary $\partial \tilde{C}_{d-1}$. Ignoring the framings, we have $\partial \tau = \sigma$ in  the relative homology group $H_{1}(\tilde{C}_{d-1}, \partial\tilde{C}_{d-1})$. However,  $h_{1}(\tilde{C}_{d-1}, \partial\tilde{C}_{d-1}) = g(C_{d-1})$, and the vertically framed cycles remain independent in $H_{1, 1}(\tilde{C}_{d-1})$, so $h_{1, 1}(\tilde{C}_{d-1}) = g(C_{d-1})$.
 
 Finally, for $h_{1, 2}(X_{d, d-1})$ consider the projection $\pi: X_{d, d-1} \longrightarrow \TP^2$. A $(1, 2)$-cycle in $X_{d, d-1}$ may give rise to a $(1, 2)$-cycle in $X_{d, d-1}$ or its image might be of lower dimension. However, 
  $h_{1, 2}(\TP^2) = 0$, see \cite{Mik08}, so the image of the cycle must be of lower dimension. Then the cycle would be entirely contained in the vertical faces of $X_{d, d-1}$, but no chains contained entirely in these faces may  be closed.  This finishes the proof of the lemma. 
\end{proof}

The following lemma uses the fact that $X_{d, d-1}$ can be obtained by a tropical modification of $\TP^2$ along non-singular curves. 

\begin{lemma}\label{lem:H11floor}
The $(1, 1)$-homology of the surface $X_{d, d-1} \subset \T(\Delta)$ is $$H_{1, 1}(X_{d, d-1}) = \Z^{d(d-1)+1}.$$ 
\end{lemma}

\begin{proof}
Consider the projection map $\pi: X_{d,d-1} \longrightarrow \TP^2$, which is given by extending the projection in the vertical direction $\R^3 \longrightarrow \R^2$ to the surface $X_{d, d-1} \subset \TP^3$. 
The claim is that   $H_{1, 1}(X_{ d,d-1})$ is generated by the parallel cycles given by  $E_k = \pi^{-1}(x_k)$ for $x_k \in C_d\cap C_{d-1}$, along with  a cycle $\tilde{L}$. 
To describe $\tilde{L}$, let    $L \subset \TP^2$ be a generic tropical line and consider $\pi^{-1}L \cap F_{d, d-1}$ this is a $(1, 1)$-chain whose boundary is a collection of vertically framed points, which occur when $L$ intersects either $C_d$ or $C_{d-1}$. Attach to these points of $\pi^{-1}L \cap F_{d, d-1}$ positively weighted rays in the $\pm e_3$ direction to obtain a closed balanced $1$-cycle $\tilde{L}$.

To show that these cycles generate the homology, take a $(1, 1)$-cycle $\alpha \in H_{1, 1}(X_{d, d-1})$ choose a representative such that all edges not contained on the floor have vertical framing and all edges contained on the floor have framing parallel to the edges of $F_{d, d-1}$. The projection $\pi(\alpha) \subset \TP^2$, can be made into a $(1, 1)$-cycle in the following way. If a supporting $1$-cell $f$ remains a $1$-cell after the projection $\pi$, then equip  $\pi(f)$ with the  framing $\pi(\phi_f)$. It is easy to check that this yields a closed $(1, 1)$-chain. Now consider the lift of $\pi(\alpha)$ just as we did for $L \subset \TP^2$, call this  lift  $\tilde{\alpha}$. If $\alpha \sim \tilde{\alpha}$ then $\alpha = k \tilde{L}$. Otherwise, $\alpha - \tilde{\alpha}$ is a non-trivial cycle, equipped only with the vertical framing so $\alpha - \tilde{\alpha}$  is a combination of $Cyc(E_k)$. In  Lemma \ref{lem:blowupsig} of the next section the intersection form is computed on the above cycles. Since the form is non-degenerate the $d(d-1) + 1$ cycles are independent in $H_{1, 1}(X_{d, d-1})$. 
This proves the lemma. 
\end{proof}

The map $\pi: X_{d, d-1} \longrightarrow \TP^2$ should be thought of as a tropical blowup at $d(d-1)$ points of sedentarity $\emptyset$. 
    Indeed it is the graph of a rational function, and it behaves as the blow up of $\TP^2$ along the common zeros of the curves $C_d, C_{d-1} \subset \R^2$.  The $1$-cycles $E_k$ are the exceptional divisors 
 and $\tilde{L}$ is the proper transform of a line. Notice that  the points of the blow up are not in general position as soon as $d >2$. 

\begin{lemma}\label{lem:ex2floor}
On the surface $X_{d, d-1}$ the difference of two cycles $E_s - E_t$ is homologous to a floor cycle. 
\end{lemma}

\begin{proof}
If $s = t$, then the difference is of course the zero cycle. If not the cycles $E_s$, $E_t$ come from distinct points $x_s, x_t \in C_d \cap C_{d-1}$. Take a path contained entirely on $C_{d-1}$ and joining the points $x_s$ and $x_t$. Similarly choose a path contained entirely on $C_{d}$ joining $x_s$ and $x_t$. Their union can be oriented to form a $1$-cycle and equipped with a vertical framing to give a $(1, 1)$-floor cycle $\alpha$, then up  to orientation $\alpha$ is homologous to $E_s- E_t$. 
\end{proof}

The above lemma provides a more convenient basis for $H_{1, 1}(X_{d, d-1})$ since the floor cycles are contained entirely in $X^o_{d, d-1}$. Choosing the appropriate differences $E_s - E_t$, we can obtain a new basis for $H_{1, 1}(X_{d, d-1})$ consisting of $\tilde{L}$ and cycles corresponding to the floor cycles in from Section \ref{subsubfloors}.

\begin{proposition}\label{lem:split}
Let $X_d$ be a  floor decomposed tropical surface of degree $d$, then the collection of cycles given in Section \ref{sec:basis} are a basis of $H_{1, 1}(X_d)$. 
\end{proposition}

\begin{proof}
When $d=1$, $H_{1, 1}(X_1) = \Z$ follows as a special case of  Lemma \ref{lem:H11floor}.
In this case, the only cycle from Section \ref{sec:basis} is $v$ which is homologous to $\tilde{L}$. 
The proposition will be proved by induction on the degree of the surface and by applying the Mayer-Vietoris sequence \ref{eq:MV} to $$X_d = X^o_{d, d-1} \cup X^o_{d-1} \qquad \text{and} \qquad \tilde{C}_{d-1} =  X^o_{d, d-1} \cap X^o_{d-1}$$  with $p=1$. 
Then along with Lemma \ref{lem:othergroups} we  obtain 
$$ \dots \longrightarrow H_{1, 2}(X_d) \longrightarrow H_{1,1}(\tilde{C}_{d-1})  \longrightarrow H_{1, 1}(X_{d-1}^o) \oplus H_{1,1}(X_{d, d-1}^o)  \longrightarrow   
$$
$$\longrightarrow H_{1, 1}(X_d)  \longrightarrow H_{1,0}(\tilde{C}_{d-1}) \longrightarrow 0. $$ 

The floor cycles from Lemma \ref{lem:ex2floor} along with the class $\tilde{L}$ from Lemma \ref{lem:H11floor} form a basis for $X_{d, d-1}$. To find $H_{1, 1}(X^o_{d, d-1})$ we use this basis and  apply the Mayer-Vietoris sequence to 
$$X_{d, d-1} = X_{d, d-1}^o \cup N_{C_{d-1}} \qquad \text{and} \qquad \tilde{C}_{d-1} = X_{d, d-1}^o  \cap N_{C_{d-1}}$$
where $N_{C_{d-1}} \subset X_{d, d-1}$ is a neighborhood of the boundary curve $C_{d-1}$. From the long exact sequence we obtain:
$$0 \longrightarrow H_{1, 1}(\tilde{C}_{d-1}) \longrightarrow H_{1, 1}(N_{C_{d-1}}) \oplus H_{1, 1}(X_{d, d-1}^o) \longrightarrow Im(j_{\ast} - k_{\ast})  \longrightarrow 0.$$
However, the map $H_{1, 1}(\tilde{C}_{d-1})  \longrightarrow H_{1, 1}(N_{C_{d-1}})$ is zero, and we obtain the sequence 
$$0 \longrightarrow H_{1, 1}(\tilde{C}_{d-1}) \longrightarrow  H_{1, 1}(X^o) \longrightarrow \Z^{d(d-1)} \longrightarrow 0.$$ 
Where $\Z^{d(d-1)}$ is generated by a collection of independent floor cycles $E_s - E_t$. This gives a map 
$l: \Z^{d(d-1)} \longrightarrow H_{1, 1}(X_{d, d-1}^o)$ for which $j_{\ast}l$ is the identity. Therefore, the above sequence is split and the following cycles form a basis for $H_{1, 1}(X_{d, d-1}^o)$: $$\{ \alpha_1, \dots , \alpha_{i(i+1)} , \gamma_1, \dots , \gamma_{g(C_{d-1})} \}.$$ 
Here the $\alpha_i$'s are independent cycles of the form $E_s - E_t$ from Lemma \ref{lem:H11floor} and the cycles $\gamma_i$ come from independent $1$-cycles in $C_{d-1}$ equipped with vertical framing. 

Also using  the induction assumption we have a basis for $X_{d-1}$. Once again by applying the Mayer-Vietoris sequence to 
$$X_{d-1} = X_{d-1}^o \cup N_{C_{d-1}} \qquad \text{and} \qquad \tilde{C}_{d-1} = X_{d-1}^o  \cap N_{C_{d-1}}$$
we obtain as a basis  of $H_{1, 1}(X_{d-1}^o)$ consisting of the given basis for $H_{1, 1}(X_{d-1})$ except without 
the vertical cycle $v$ and along with vertically framed cycles $\{ \gamma_1, \dots , \gamma_{g(C_{d-1})} \}$ similar to above.  

Now the map $ H_{1,1}(\tilde{C}_{d-1})  \longrightarrow H_{1, 1}(X_{d-1}^o) \oplus H_{1,1}(X_{d, d-1}^o)  $ is given by $\gamma_i \mapsto (\gamma_i, \gamma_i)$, and so it is injective.  Therefore, $H_{1, 2}(X_d) = 0$, moreover the quotient $$\frac{H_{1,1}(X^o_{d-1, d}) \oplus H_{1, 1}(X^o_{d-1})}{i_{\ast}H_{1, 1}(\tilde{C}_{d-1})} $$ is torsion free. 
Section \ref{subsubjoin} constructed $(1, 1)$-cycles in $X_d$ from $(1, 0)$-cycles in $\tilde{C}_i$ for all $1\leq i < d$. This  gives a map $l : H_{1, 0}(\tilde{C}_i) \longrightarrow H_{1, 1}(X_d)$. We claim that  $l$ is  injective.  A $(2, 1)$-chain $\tau$ bounding $l(\sigma)$ for some $\sigma \in H_{1,0}(\tilde{C}_i)$ can be intersected with a hyperplane $z_3 = c$ for some constant $c$. For a generic choice of $c$ the intersection is $1$-dimensional and produces a $(1, 1)$-chain in $\tilde{C}_{d-1}$ bounding $\partial l(\sigma) = \sigma$. Thus $\sigma \sim 0$ and the map is injective. 

Moreover, the composition $\partial_{\ast} l$ is the identity on $H_{1, 0}(\tilde{C}_i)$ and thus  the following short exact  sequence is split
$$0 \longrightarrow  \frac{H_{1,1}(X^o_{d-1, d}) \oplus H_{1, 1}(X^o_{d-1})}{i_{\ast}H_{1, 1}(\tilde{C}_{d-1})} \longrightarrow H_{1, 1}(X) \longrightarrow H_{1, 0}(\tilde{C}_{d-1}) \longrightarrow 0.$$
Therefore there is the following isomorphism 
$$ H_{1,1}(X_d) \cong \frac{H_{1,1}(X^o_{d-1, d}) \oplus H_{1, 1}(X^o_{d-1})}{i_{\ast}H_{1, 1}(\tilde{C}_{d-1})} \oplus H_{1, 0}(\tilde{C}_{d-1}), $$
which  proves that the claimed collection of cycles forms a basis. 
\end{proof}

\vspace{0.5cm}

\begin{corollary}\label{cor:rkH11}
For $X_d \subset \TP^3$ a non-singular floor decomposed surface we have 
$$h_{1, 1}(X_d)  = \frac{2d^3 - 6d^2 + 7 d}{3}.$$
\end{corollary}

\begin{proof}
Following Proposition \ref{lem:split} it suffices to determine the size of the basis. 
Firstly, there are $i^2 + i -1$ cycles of the type from Section \ref{subsubfloors}  on a floor $F_{i+1, i}$. Therefore in total these contribute $$\frac{d^3 - 4d + 3}{3} =  \sum_{i=1}^{d-1} i^2 + i -1 $$ cycles to the basis. 
The next two types of cycles come in pairs and there are exactly $b_2(X_d)$ such pairs. So in total these contribute 
$$\frac{d^3 -6d^2 +11d - 6}{3}.$$
In addition there is the cycle $v$. Combining this we obtain the claimed rank of $H_{1, 1}(X_d)$. 
\end{proof}

The next corollary follows from a direct substitution. 

\begin{corollary} \label{cor:rankd-d-1}
For $X_d \subset \TP^3$ a non-singular floor decomposed surface we have 
 $$ h_{1, 1}(X_{d}) = h_{1, 1}(X_{d-1})  + d(d-1) + 2g(C_{d-1}) -1.$$
\end{corollary}

\end{section}

\begin{section}{The intersection form on $H_{1, 1}$}\label{sec:form}

The intersection form for a floor decomposed surface can be determined by induction on $d$ and the form on the surface $X_{d, d-1}$. 

\begin{lemma}\label{lem:blowupsig}
The intersection form on a non-singular tropical surface $X_{d, d-1} \subset \TT(\Delta)$ has signature $1-d(d-1)$. 
\end{lemma}

\begin{proof}Lemma \ref{lem:H11floor} gives generators for the $(1, 1)$-homology of $X_{d, d-1}$, 
they are $\tilde{L}, E_1, \dots , E_k$ for $k = d(d-1)$. All of these cycles are pairwise disjoint, so it remains to calculate their self-intersections. 
A cycle $E_i$ passes through a single vertex of the surface $X_{d, d-1}$, exactly as the cycle $\alpha$ in Figure \ref{fig:-1int}. As mentioned in Example \ref{negInt} the two cycles in this figure are homologous in $X_{d, d-1}$. So the self-intersection is calculated in Example \ref{negInt} and is $E_i^2 = -1$ for $1 \leq i \leq d(d-1)$. 

To find a cycle homologous to $\tilde{L}$ and intersecting it transversally it suffices to take a generic  translation of $L$, $L^{\prime} \subset \TP^2$ and lift it to $X_{d, d-1}$ as in the proof of Lemma \ref{lem:H11floor}. Therefore, $\tilde{L}^2 = 1$. 
This proves the lemma. 
\end{proof}

\begin{proof}[Proof of Theorem \ref{thm:H11floor}]
Proposition \ref{lem:split} gives a basis of $H_{1, 1}(X_d)$, so it 
 remains only to find the signature. 
By this same proposition $H_{1, 1}(X_d)$ is a direct sum of the groups $A_1, A_2, C, B$ and $\langle v \rangle$, where:
\begin{itemize}
\item $A_1$ is the subspace generated by  cycles contained entirely in $X_{d-1}^o$ and not in $H_{1, 1}(\tilde{C}_{d-1})$, $$A_1 \cong \frac{H_{1, 1}(X^o_{d-1})}{i_{\ast}H_{1, 1}(\tilde{C}_{d-1})}.$$
\item $A_2$ is the subspace generated by the floor cycles in $X_{d, d-1}^o$, 
$$A_2 \cong \frac{H_{1, 1}(X^o_{d, d-1})}{i_{\ast}H_{1, 1}(\tilde{C}_{d-1})}.$$
\item  $C$ is the subspace generated by the cycles $\gamma_i$ from $H_{1, 1}(\tilde{C}_{d-1})$ and $B$ is the subspace generated by their pairs  $\beta_i$. 
\item  $v$ is the $(1, 1)$-cycle arising from the vertically framed point in $H_{1, 0}(\tilde{C}_{d-1})$.
\end{itemize}
Recall the cycles from $C$ and $B$ arise in pairs. Let  $\beta$, $\gamma$, be such a pair of dual $(1, 1)$-cycles. 
The intersection form restricted to this pair is:
 $$
 \left(\begin{array}{cc}0 & 1 \\1 & \star \end{array}\right).$$ Otherwise the $\beta$ and $\gamma$ classes can be taken so that they are disjoint.  Also, the classes in $C$ are all pairwise disjoint, moreover, $C \bot A_1$, $C \bot A_2$. We also have  $A_1 \bot A_2$ and $v$ is orthogonal to $A_1, A_2, C$ and $B$. 
Therefore, the  intersection form on $H_{1, 1}(X_d)$ has the form:

\begin{center}
  \begin{tabular}{  c || c | c | c | c| c| }
&   $A_1$ & $A_2$ & $v$ & $B$ & $C$  \\
    \hline \hline 
   $A_1$ &  $\star$  & 0 & $ 0 $ & $\star$ &  0     \\ \hline
   $A_2$ &  0 & $\star$  & $ 0 $  & $\star$ & 0   \\ \hline
   $v$ &  $0 $  & $0 $  & $\star$  &$\star  $ & 0    \\ \hline
   B &$\star$ &$\star$  & $\star$ & $\star$ & $\begin{array}{ccc} 1&  &  \\ & \ddots &  \\ &  &1 \end{array}$ \\ \hline
  $C$ & 0 & 0 &  0& $\begin{array}{ccc}1&  &  \\ & \dots &  \\ &  & 1\end{array}$ &  0  \\ \hline 
  \end{tabular}
\end{center}
The form restricts to a form with index zero on the subspace $B \oplus C$. 
This is exactly the same situation as Novikov additivity of the signature for glueing  $4$-manifolds, see Theorem 5.3 of  \cite{Kir}. Just as in that case we obtain:
 $$Sign(X) = 1 + Sign(A_1) + Sign(A_2).$$
The form is negative definite on $A_2$   since it has a  basis given by floor cycles, as described by Lemma \ref {lem:ex2floor}, 
 so
$$Sign(A_2) = d(d-1) -1.$$ Now the space $A_1$ with intersection product is the same as  the orthogonal complement of the hyperplane section in $H_{1, 1}(X_{d-1}).$  Again for $d=1$ it can be verified that the form on $\Z = H_{1, 1}(X_1)$ is positive definite. Using the above and applying induction we have:
$$Sign(A_1) = 2 + 2b_2(X_{d-1}) - h_{1, 1}(X_{d-1}) .$$

Combining the above three equalities we obtain 
\begin{align}
Sign(X) = 2  + 2b_2(X_{d-1}) - h_{1, 1}(X_{d-1}) + d(d-1),
\end{align}
which reduces to $Sign(X) = 2 + 2b_2(X_{d}) - h_{1, 1}(X_{d})$ after substituting 
$b_2(X_d) = b_2(X_{d-1}) + b_1(C_{d-1})$ and applying Corollary \ref{cor:rankd-d-1}.  
This completes the proof of the theorem. \end{proof}

Floor decomposed surfaces are an instance of a  more general operation, known as the tropical sum of surfaces constructed in \cite{KMS}. This  construction allows two compatible tropical surfaces to be glued to yield a new one, see \cite{KMS} for more details.
The next corollary follows from the above proof and Lemma \ref{lem:blowupsig}. It  says the signature of the intersection forms of $X_d$, $X_{d-1}$, and $X_{d, d-1}$ are additive,  similar to Novikov additivity of the signature under connect sums, see \cite{Kir}.

\begin{corollary}\label{cor:addsig}
Let $X_d \subset \TP^3$ be a smooth tropical floor decomposed surface then, 
$$Sign(X_d) = Sign(X_{d-1}) + Sign( X_{d, d-1}),$$ where
 $X_{d-1}  \subset \TP^3$ and $X_{d, d-1} \subset \T(\Delta)$. 
\end{corollary}
\end{section}

The difference in the tropical and classical intersection forms is to be expected given  the relation of the tropical $(p, q)$-homology groups of $X_d$ to a filtration of $H^2(\X_d)$ for a  non-singular complex surface $\X \subset \CP^3$ of the same degree, see  \cite{ItKaMiZh}.  
There is also a version of the Hodge index theorem for algebraic cycles, see Section 5.1 of \cite{Hart}.  It says that the intersection pairing  restricted to algebraic cycles of a non-singular projective surface $\X$ has signature $(1, r-1)$ where $r$ is the rank of the Picard group of $\X$. It is so far unknown if the analogous statement holds for the group of tropical $1$-cycles  of a non-singular tropical surface $X$ modulo rational equivalence.

\small
\def\rightmark{\em Bibliography}

\bibliographystyle{alpha}
\bibliography{Biblio.bib}

\end{document}